\documentclass[10pt]{amsart}

\usepackage[bbgreekl]{mathbbol}
\usepackage{pdflscape}
\usepackage{xypic}
\usepackage{amssymb, amsmath, amsthm, amsfonts}
\usepackage{mathrsfs,comment}
\usepackage{graphicx}
\usepackage{placeins}
\usepackage{rotating}
\usepackage{tikz}
\usepackage{float}
\usepackage{hvfloat}
\usepackage{caption}
\usepackage{pdflscape}
\usepackage[pagebackref]{hyperref}
\usepackage{url}
\usepackage[all,arc,2cell]{xy}
\UseAllTwocells
\usepackage{enumerate}
\usepackage{chngcntr}
 \usepackage{lineno}
 \usepackage{blindtext}
\usepackage{verbatim}
\usepackage{soul}
\usepackage[normalem]{ulem}
\usepackage{stmaryrd}

\usepackage[nameinlink,capitalise,noabbrev]{cleveref}

\hypersetup{%
  bookmarksnumbered=true,%
  bookmarks=true,%
  colorlinks=true,%
  linkcolor=blue,%
  citecolor=blue,%
  filecolor=blue,%
  menucolor=blue,%
  pagecolor=blue,%
  urlcolor=blue,%
  pdfnewwindow=true,%
  pdfstartview=FitBH}

\def\@url#1{{\tt\def~{\lower3.5pt\hbox{\char'176}}\def\_{\char'137}#1}}

\makeatletter
\newtheorem*{rep@thm}{\rep@title}
\newcommand{\newrepthm}[2]{%
\newenvironment{rep#1}[1]{%
 \def\rep@title{#2 \ref{##1}}%
 \begin{rep@thm}}%
 {\end{rep@thm}}}
\makeatother

\newrepthm{thm}{Theorem}

\newtheorem{thm}{Theorem}[section]
\newtheorem{cor}{Corollary}[section]
\newtheorem{prop}{Proposition}[section]
\newtheorem{lem}{Lemma}[section]
\newtheorem*{thm*}{Theorem}

\theoremstyle{definition}
\newtheorem{defn}{Definition}[section]

\theoremstyle{remark}
\newtheorem{rem}{Remark}[section]
\newtheorem{example}{Example}[section]

\makeatletter
\let\c@lem=\c@thm
\let\c@cor=\c@thm
\let\c@prop=\c@thm
\let\c@lem=\c@thm
\let\c@defn=\c@thm
\let\c@exmps=\c@thm
\let\c@rem=\c@thm
\let\c@warn=\c@thm
\let\c@claim=\c@thm
\let\c@quest=\c@thm
\let\c@notation=\c@thm
\let\c@note=\c@thm
\let\c@example=\c@thm
\makeatother

\numberwithin{equation}{section}

\newcommand{\Z}{\mathbb{Z}}

\newcommand{\Q}{\mathbb{Q}}

\newcommand{\F}{\mathbb{F}}

\newcommand{\G}{\mathbb{G}}
\newcommand{\W}{\mathbb{W}}
\newcommand{\Sn}{\mathbb{S}}

\DeclareSymbolFontAlphabet{\scr}{rsfs}

\newcommand{\smsh}{\wedge}

\newcommand{\xra}{\xrightarrow}

\def\quickop#1{\expandafter\newcommand\csname #1\endcsname{\operatorname{#1}}}
\quickop{Hom} \quickop{End} \quickop{Aut} \quickop{Tel} \quickop{Mic} 
\quickop{Ext} \quickop{Tor} \quickop{Cotor} \quickop{Id} \quickop{Coker} \quickop{Ker}
\quickop{Lim} \quickop{Colim} \quickop{Holim} \quickop{Hocolim}
\quickop{id} \quickop{tel} \quickop{mic} \quickop{coker} 
\quickop{colim} \quickop{holim} \quickop{hocolim} \quickop{im}
\DeclareMathOperator{\Gal}{Gal}

\DeclareMathOperator{\Pic}{Pic}

\DeclareMathOperator{\Map}{Map}
\DeclareMathOperator{\map}{map}

\newcommand{\Kn}{{\mathbf{K}}} 
\newcommand{\E}{{\mathbf{E}}} 
 
\newcommand{\Sdet}{S\langle {\det} \rangle}

\newcommand{\Zdet}{\mathbb{Z}_{p}\langle {\det} \rangle}
\newcommand{\DDet}{\langle \det\rangle}

\newcommand{\haut}{\mathrm{haut}}

\usepackage{todonotes}

\newcommand{\alg}{\mathrm{alg}}

\title{Constructing the determinant sphere using a Tate twist}
\author[Barthel]{Tobias Barthel}
\address[Barthel]{Max Planck Institute for Mathematics}
\email{tbarthel@mpim-bonn.mpg.de}
\author[Beaudry]{Agn\`es Beaudry}
\address[Beaudry]{Department of Mathematics, University of Colorado Boulder}
\email{agnes.beaudry@colorado.edu}
\author[Goerss]{Paul G. Goerss}
\address[Goerss]{Department of Mathematics, Northwestern University}
\email{pgoerss@math.northwestern.edu}
\author[Stojanoska]{Vesna Stojanoska}
\address[Stojanoska]{Department of Mathematics, University of Illinois at Urbana-Champaign}
\email{vesna@illinois.edu}
\thanks{This material is based upon work supported by the National Science Foundation under Grant No.~DMS-1812122 and Grant No.~DMS-1725563. Barthel was partially supported by the DNRF92 and the European Unions Horizon 2020 research and innovation programme under the Marie Sklodowska-Curie grant agreement No.~751794. 
The authors would like to thank the Isaac Newton Institute for Mathematical Sciences for support and hospitality during the program Homotopy Harnessing Higher Structures when work on this paper was undertaken. This work was supported by EPSRC Grant Number EP/R014604/1.}

\newcommand{\SG}{\mathrm{S}\G}
\newcommand{\ST}{S(1)}

\newcommand{\K}{\mathbb{K}} 
\newcommand{\mm}{\mathfrak{m}} 

\newcommand{\Func}{F_c}

\newcommand{\MJ}[1]{\mathrm{M}_{J(#1)}}

\newcommand{\FF}{\F}
\newcommand{\GG}{\G}
\newcommand{\WW}{\W}
\newcommand{\ZZ}{\Z}
\newcommand{\Gl}{{\mathrm{Gl}}}
\newcommand{\longr}{\longrightarrow}
\newcommand{\bdet}{\langle\det\rangle}

\newcommand{\STa}{\widetilde\ST}

\newcommand{\mmu}{\bbmu}
\newcommand{\Cc}{C}

\begin{document}

\maketitle

\begin{abstract} 
Following an idea of Hopkins, we construct a model of the determinant sphere $\Sdet$ in the category of $K(n)$-local spectra. To do this, we build a spectrum which we call the Tate sphere $S(1)$. This is a $p$-complete sphere with a natural continuous action of $\Z_p^\times$.  The Tate sphere inherits an action of $\mathbb{G}_n$ via the determinant and smashing Morava $E$-theory with $S(1)$ has the effect of twisting the action of $\mathbb{G}_n$. 
A large part of this paper consists of analyzing continuous $\mathbb{G}_n$-actions and their homotopy fixed points in the setup of Devinatz and Hopkins. 
\end{abstract}

\section{Introduction}

Let $p$ be a prime and  $n > 0$ an integer; these will be fixed throughout and we will always suppress $p$ and
mostly suppress $n$ from the notation. Let $\E =E_n$ denote the Lubin--Tate spectrum associated to the Honda formal group law of height $n$ over $\FF_{p^n}$, and let $\Kn=K(n)$ be the corresponding Morava $K$-theory at height $n$ at the prime $p$. As is the usual convention, given any spectrum $X$, we write
\[
\E_\ast X = \pi_\ast L_{\Kn}(\E \smsh X)
\]
where $L_{\Kn}$ denotes $\Kn$-localization. 

We are interested in the $\Kn$-local category and, in particular, one very interesting 
spectrum therein which arises from comparing two dualities. The first of these duality functors is Spanier--Whitehead duality, sending $X$ to $D_nX=F(X,L_{\Kn}S^0)$. If $X$ is a dualizable
spectrum -- for example if $X$ is a finite spectrum --  then $\E_\ast D_nX \cong \E^{-\ast}X$ and can be
computed by a universal coefficient spectral sequence. The second is Gross--Hopkins duality, sending
$X$ to $I_nX =F(M_nX, I_{\Q/\Z})$, the Brown--Comenetz dual of its monochromatic layer. Specifically, $M_nX$ is the fiber of $L_nX \to L_{n-1}X$ and $I_{\Q/\Z}$ is the spectrum representing the cohomology theory $I_{\Q/\Z}^*(X)=\Hom_{\Z}(\pi_*X, \Q/\Z)$. 
It is a consequence of the work of Gross and Hopkins that the dual $I_n$ of
the sphere $L_{\Kn}S^0$ is invertible in the $\Kn$-local category and, hence, we have for any spectrum $X$ a natural equivalence
\[
I_nX \simeq L_{\Kn}(D_nX \smsh I_n).
\]

At this point, information about the homotopy type of $I_n$ becomes vital, and one gets a handle on it using that the spectrum $\E$ has an action by the Morava
stabilizer group $\GG = \GG_n$. Consequently, the graded $\E_\ast$-module $\E_\ast X$ has
a continuous action by $\GG$, giving it the structure of a \emph{Morava module} (see Definition~5.3.20 \cite{BarthelBeaudry}).

The key to the invertibility of $I_n$ is the calculation of the Morava module $\E_\ast I_n$. The group $\GG$ is a semidirect product $\Sn \rtimes \Gal(\F_{p^n}/\FF_p)$, where $\Sn=\Sn_n$ is the automorphism group of the formal group law of $\Kn$. The group $\Sn$ can be identified with a subgroup of the general linear group $\Gl_n(\WW)$, where $\WW$ denotes the Witt vectors on the finite field $\FF_{p^n}$. The group $\Sn$ has enough symmetry that the 
determinant $\Gl_n(\WW) \to \WW^\times$ restricts to a homomorphism
\[
\det\colon\Sn \longrightarrow \ZZ_p^\times,
\]
which can be extended to $\GG$ as the composite
\[ \det\colon\GG = \Sn \rtimes \Gal(\FF_{p^n}/\FF_p) \xrightarrow{\det \times \id} \ZZ_p^\times \times \Gal(\FF_{p^n}/\FF_p) \xrightarrow{proj_1} \ZZ_p^\times. \]

This gives a $\GG$-action on $\ZZ_p$, and we write the corresponding representation as $\Zdet$. If $M$ is a Morava module, we can 
define a new Morava module by $M\langle {\det} \rangle = M \otimes_{\ZZ_p} \Zdet$ with the diagonal $\GG$-action. Then
we have by \cite{grosshopkins1} and \cite{StrickGrossHop} an isomorphism of Morava modules
\[
\E_\ast I_n = \E_\ast(S^{n^2-n})\DDet.
\]
If the prime is large ($2p > \mathrm{max}\{n^2+1,2n+2\}$) this determines the homotopy type of $I_n$. If the prime
is not large, then we would like a fixed model $\Sdet$ of an invertible spectrum in the $\Kn$-local category equipped
with an isomorphism
\[
\E_\ast \Sdet \cong \E_\ast\DDet.
\]
Then we have a $\Kn$-local equivalence
\[
I_n \simeq S^{n^2-n} \smsh \Sdet \smsh P_n,
\]
where $P_n$ is an invertible $\Kn$-local spectrum with $\E_\ast P_n \cong \E_\ast S^0$ as Morava modules, and attention
turns to identifying $P_n$. In the known cases this comes down to calculating the homotopy groups of $I_nX$ for $X$ a particularly nice type $n$ complex. See \cite{GHM_pic} for analysis of $P_n$ at $n=2=p-1$; the case $n=1=p-1$ was done by \cite{HMS_pic} and also appears in \cite{HS-KRD,GHM_pic}.

The point of this note is to give a construction of a model of $\Sdet$ valid at all primes $p$ and all $n > 0$. 
We actually give two constructions of $\Sdet$, one using homotopy fixed points, following an idea of Mike Hopkins, and another, 
more naive and direct one, following ideas from \cite{GHM_pic,craig_imj}, fixing the typos therein and extending the construction to the prime $2$. A different construction of $\Sdet$, valid at primes large with respect to the given height and choice-free, was given by Peterson in \cite[Cor.~3]{PetersonEric}. Since the Morava module determines an invertible $\Kn$-local object at large primes, the two constructions give equivalent spectra in this situation.

The first model will evidently have the property that $L_{\Kn}(\E^{h\K} \smsh \Sdet) =
\E^{h\K}$ for all closed subgroups $\K$ in the kernel of the determinant. The key to this construction is to introduce a spectrum $\ST$ with a continuous $\G$--action, non-equivariantly equivalent to the $p$-complete sphere spectrum $S^0=S^0_p$, and such that smashing with it naturally twists $\G$--actions by the determinant representation. Then we define 
\[ \Sdet = (\E \smsh \ST)^{h\GG},\]
the action on the right-hand side being diagonal. The following is our main result. 

\begin{repthm}{thm:pretty-damn-cool} 
There is a canonical $\GG$-equivariant equivalence $f\colon \E \smsh \Sdet \to \E \smsh \ST$, where the action of $\GG$ on the source is via the action on $\E$, while on the target it is diagonal. This induces an isomorphism of Morava modules $\E_*\Sdet \cong \E_*{\left<\det \right>}$.
\end{repthm}

If $\K$ is a closed subgroup of $\G$ in the kernel of the determinant, taking $\K$-homotopy fixed points in this equivalence gives the desired result (\cref{cor:invariance})
\[ \E^{h\K} \smsh \Sdet \simeq (\E \smsh \ST )^{h\K} \simeq \E^{h\K}.\]

This project gives a chance to revisit and give an encomium on the amazing paper of Devinatz and Hopkins on 
fixed point spectra in the $\Kn$-local category \cite{DH}. Distilled down we have the following question: let $X$ be
a spectrum with a continuous action of the Morava stabilizer group $\GG$. We can then form the
$\GG$-spectrum $Z= \E \smsh X$ with diagonal $\GG$-action and discuss the homotopy type of
$Z^{h\GG}=(\E \smsh X)^{h\GG}$. Note that $\E_\ast Z=\pi_\ast L_{\Kn}(\E \smsh Z)$ has two
$\GG$-actions: the Morava module action on $\E$ and the action on $Z$. A consequence of our results is that
if $X$ is dualizable in the $\Kn$-local category, then
\begin{equation}\label{eq:maineq}
\E_\ast\left(Z^{h\GG}\right) = \E_\ast (\E \smsh X)^{h\GG} \cong \E_\ast X
\end{equation}
and the Morava module action on $\E_\ast (\E \smsh X)^{h\GG}$ corresponds to the diagonal action on
\[
\E_*X = \pi_\ast L_{\Kn}(\E \smsh X).
\]
An analogue of this result for arbitrary spectra $X$ with \emph{trivial} $\GG$-action was proven by Davis and Torii~\cite{davistorii}. The equivalence \eqref{eq:maineq} is not hard to prove once we have come to terms with the notion of a continuous
$\GG$-action. Since we are making a homology calculation we need cosimplicial techniques, and this is exactly what
Devinatz and Hopkins supply. 

We close with a remark on our choice of the formal group we use to specify Morava $K$-theory and $E$-theory. At the beginning of this introduction, we specified the Honda formal group over $\FF_{p^n}$. This was simply because \cite{DH} is written for the Honda formal group. Presumably, the work of Devinatz and Hopkins goes through without change for any height $n$ formal group over any finite extension $\FF_q$ of the prime field $\FF_p$. If this is the case, we could choose any $F$ so that the map $\det\colon\Aut(F/\FF_q) \to \ZZ_p^\times$ is surjective.

\subsection*{Acknowledgements}

We would like to thank Hans-Werner Henn, Mike Hopkins and Charles Rezk for helpful conversations and the referee for their comments.

\section{Continuous $\G$ actions and their homotopy fixed points}\label{sec:continuousactions}

As is perhaps apparent from the introduction, we will assume our readership has access to the standard framework
of $\Kn$-local homotopy theory. The usual source for an in-depth study of the technicalities
is Hovey and Strickland \cite{hovstrmemoir} and basic introductions can be found in almost any paper on chromatic homotopy 
theory.  We were especially thorough in \cite[\S 2]{BGH}. 

Less familiar is the analysis of point-set properties of the action of Morava stabilizer group $\GG$ on the spectrum $\E$. We
will need to use an explicit construction of the homotopy fixed points. For our purposes the original definition by
Devinatz and Hopkins~\cite{DH} will do. The reader interested in extensions and variations of the original notion may want to consult 
work such as Behrens--Davis \cite{behrensquick}, Davis--Quick  \cite{davisquick} and Quick \cite{quick2}.

We will also not access the full power and structure of equivariant stable
homotopy theory. 
Our $G$-spectra will simply be $G$-objects in some suitable category of spectra; when $G$ is profinite, we will also use a simple notion of continuity (see \cref{def:cont-g-spec}).

We start with some algebra. Recall that $\E_\ast = \WW\llbracket u_1,\ldots,u_{n-1}\rrbracket[u^{\pm 1}]$ where the power series ring is in
degree zero and the degree of $u$ is $-2$ and let $\mm \subset \E_0$ be the maximal ideal. 

\begin{rem}\label{rem:compl-tower} 
Before we proceed further, we need to establish some more notation. Using the periodicity results of Hopkins and Smith \cite{nilpotence2}, Hovey and
Strickland produce a sequence of ideals
$J(i) \subseteq \mm \subseteq \E_0$ and finite type $n$ spectra $\MJ{i}$ with the following properties:
\begin{enumerate}

\item $J(i+1) \subseteq J(i)$ and $\bigcap_i J(i) = 0$;

\item $\E_0/J(i)$ is finite;

\item $\E_0(\MJ{i}) \cong \E_0/J(i)$ and there are spectrum maps $q\colon \MJ{i+1} \to \MJ{i}$ realizing the quotient
$\E_0/J(i+1) \to \E_0/J(i)$;

\item There are maps $\eta=\eta_i\colon S^0 \to \MJ{i}$ inducing the quotient map $\E_0 \to \E_0/J(i)$ and $q\eta_{i+1} = \eta_i\colon
S^0 \to \MJ{i}$; 

\item If $X$ is a finite type $n$ spectrum, then the map $X \to \holim_i(X \smsh \MJ{i})$ induced by the maps
$\eta$ is an equivalence.

\item  If $X$ is any $L_n$-local spectrum then by \cite{hovstrmemoir} we have $L_{\Kn}X \simeq \holim_i X \smsh \MJ{i}$. In 
particular we have $ \E \simeq \holim_i \E \smsh \MJ{i}$.

\end{enumerate} 
Most of this is proved in \cite[\S~4]{hovstrmemoir}, and (6) is proved in \cite[Prop.~7.10]{hovstrmemoir}.
Hovey and Strickland also prove that items (1)-(5) characterize the tower $\{\MJ{i}\}$ up to equivalence
in the pro-category of towers under $S^0$. See Proposition 4.22 of \cite{hovstrmemoir}.
Note that the sequence $\{ J(i)\}$ of ideals defines the same topology on $\E_0$ as the $\mm$-adic topology and
that $\G$ acts on $\E_0/J(i)$ through a finite quotient. 
\end{rem}

For profinite sets $T=\lim_j T_j$ and $A=\lim_i A_i$, recall that the set of continuous maps from $T$ to $A$ is defined as
\begin{equation*}
\Map^c(T,A) = {\lim}_i \colim_j \Map(T_j,A_i).
\end{equation*} 
Let $M$ be a Morava module and always assume $M$ is $\mm$-complete. An important example of the previous construction is the Morava module of continuous maps 
\begin{equation*}
\Map^c(\G,M) = {\lim}_i \Map^c(\G,M/\mm^i) = {\lim}_i \colim_j \Map(\G/U_j,M/\mm^i)
\end{equation*} 
where $U_{j+1} \subseteq U_j \subseteq \G$ is a nested sequence of open normal subgroups so that $\cap\  U_j = \{e\}$; then
$\G = \lim_j \G/U_j$.

We now begin to make these constructions topological by giving a definition of a spectrum of continuous maps in the $\Kn$-local category.
\begin{defn}\label{rem:func-defined}
Suppose  $T = \lim_j T_j$ is a profinite set, and
$A \simeq \holim_i A\smsh \MJ{i}$ is a $\Kn$-local spectrum. Define
\[
\Func(T_+,A) = \holim_i \hocolim_j F( {T_j}_+, A\smsh \MJ{i}). 
\] 
\end{defn}

In applications $T$ will be $\GG$ or $\GG/\K \times \GG^s$ with $s \geq 0$ and $\K \subseteq \GG$ a closed subgroup, or $\GG=\Z_p^{\times}$.

We now calculate $\pi_\ast \Func(T_+,A)$, at least for some $A$. For later applications,
we will need a slightly more general result about $\pi_\ast F(Z,\Func(T_+,A))$ with $Z$ arbitrary. If $Z$ is any spectrum we may write $Z \simeq \hocolim Z^\alpha$ for some filtered collection of finite spectra. If
$A \simeq \holim_i A\smsh \MJ{i}$ is a $\Kn$-local spectrum, then we have a topology on
$\pi_t F(Z,A) = A^{-t}(Z)$ defined by the open system of neighborhoods of zero given by the kernels of the map 
\[
\pi_t F(Z,A) \longrightarrow \pi_tF(Z^\alpha,A \smsh \MJ{i}).
\]
This is the {\it natural topology} of \cite[Section 11]{hovstrmemoir}. The groups $\pi_\ast F(Z,A)$ are complete
in this topology if
\[
\pi_\ast  F(Z,A) \cong \lim_{\alpha,i} \pi_*F(Z^\alpha,A \smsh \MJ{i}).
\]
In applying the following result our main example will be $A=\E\smsh X$ with $X$ dualizable in the $\Kn$-local category.

\begin{lem}\label{lem:ContinuousMaps} 
Suppose $Z$ is any spectrum, $T = \lim_j T_j$ is a profinite set, and
$A \simeq \holim_i A\smsh \MJ{i}$ is a $\Kn$-local spectrum. Further suppose $\pi_t (A \smsh \MJ{i})$ is finite
for all $i$ and $t$. We then have an isomorphism
\[
\pi_* F(Z, \Func(T_+, A)) \cong \Map^c(T, A^{-*}Z)
\]
where $A^{-*}Z$ is equipped with the natural topology.
\end{lem}

\begin{proof} Let $Z \simeq \hocolim_{\alpha} Z^\alpha$ be some cellular filtration on $Z$ by finite spectra. Our
finiteness hypothesis on $A$ implies
\[
A^{-*}Z = \pi_*F(Z,A) \cong \lim_{\alpha, i} \pi_* F(Z^\alpha, A\smsh \MJ{i}).
\]
Now we have that
\[
F(Z, \Func(T_+, A)) \simeq \holim_{\alpha} \holim_i F(Z^\alpha, \hocolim_j F({T_j}_+, A\smsh \MJ{i}))
\]
is equivalent to 
\[ 
\holim_{\alpha} \holim_i   \hocolim_j  F(Z^\alpha,F({T_j}_+, A\smsh \MJ{i})),
\]
since $Z^\alpha$ is dualizable. The homotopy groups of 
\[
F(Z^\alpha,F({T_j}_+, A\smsh \MJ{i})) \simeq F({T_j}_+, F(Z^\alpha, A\smsh \MJ{i}))
\]
are $\Map(T_j, \pi_* F(Z^\alpha, A\smsh \MJ{i}))$ and the claim follows using the Milnor sequence and our finiteness hypotheses for the vanishing of the ${\lim}^1$ term.
\end{proof}

\begin{rem}\label{rem:func} For a $\Kn$-local spectrum $X \simeq \holim X \smsh \MJ{i}$, we can give 
\[F((\GG/U_j)_+,X \smsh \MJ{i})\] a left $\GG  = \lim_i \GG/U_i$ action
by operating on the right on the source. (Note that the subgroups $U_j$ are normal.) This assembles into an action on $\Func(\GG_+,X)$. If the homotopy groups $\pi_t (X \smsh \MJ{i})$ are finite, 
\Cref{lem:ContinuousMaps} gives an isomorphism of continuous $\GG$-modules
\begin{equation}\label{eq:alg-func}
\pi_\ast \Func(\GG_+,X) \cong \Map^c (\GG,\pi_\ast X)
\end{equation}
where again $\GG$ acts on the source.

Writing
$\GG^s = \lim (\GG/U_i)^s$ we define $\Func(\GG_+^s,X)$ for $s \geq 1$ as in \Cref{rem:func-defined}. We have that
\[
\Func(\GG_+^s,\Func(\GG_+^t,X)) \simeq \Func(\GG_+^{s+t},X).
\]
The equation $\Func(\GG_+^{s+1},X) \simeq \Func(\GG_+,\Func(\GG_+^s,X))$ defines an action of $\GG$ on $\Func(\GG_+^{s+1},X)$ 
using the right action on the first factor of $\GG^{s+1}$.

Evaluation defines a map $\GG_+ \smsh F((\GG/U_j)_+,X \smsh \MJ{i}) \to X \smsh \MJ{i}$. Here $\GG$ is simply regarded
as a set, with no topology. These fit together to give a map
\[
\GG_+ \smsh \Func(\GG_+,X) \longr X.
\]
\end{rem}

We now come to the Devinatz--Hopkins notion of a continuous $\GG$--action on a $\Kn$-local spectrum. To prepare,
we spend a few paragraphs examining the standard bar construction in equivariant homotopy theory.

Let $G$ be a discrete group and $X$ a $G$-space. Then we can form the augmented cosimplicial space
\begin{equation}\label{eq:basic-cosimpl}
X \longrightarrow \map(G^{\bullet + 1},X)
\end{equation}
with coface maps defined by
\[
(d^i\phi)(g_0,g_1,\ldots,g_s) =
\begin{cases}
g_0\phi(g_1,\ldots,g_s),& s=0;\\
\phi(g_0,\ldots,g_ig_{i+1},g_s),& s \geq 1.
\end{cases}
\]
The codegeneracy maps $s^i$ insert the unit in the $i$th slot and the augmentation $\eta\colon X \to \map(G,X)$ 
is adjoint to the action map $G \times X \to X$.  Notice that $s^i$ for all $i$, and $d^i$ for all 
$i \geq 1$ depend only on $G$, and not on the action of $G$ on $X$. However, 
for all $s$, $d^0$ is given by the composition
\[
\xymatrix{ 
\map(G^s,X) \ar[rr]^-{\map(G^s,\eta)} && \map(G^s,\map(G,X)) \ar[r]^-\cong & \map(G^{s+1},X),
}
\]
where we have used the adjoint isomorphism $\map(Y,\map(G,X)) \cong \map(G \times Y,X)$. 

We could turn these observations around and {\it define} a $G$-action on $X$ as a map $\eta\colon X \to \map(G,X)$ 
so that the diagram \eqref{eq:basic-cosimpl} determined by these formulas is an augmented
cosimplicial space; that is, the various compositions satisfy the cosimplicial identities. We find that this is 
equivalent to the usual definition. 

There is nothing special about spaces in this discussion:for example, an action of $G$ on a spectrum $X$ defines
and is defined by an augmented cosimplicial spectrum
\begin{equation*}
X \longrightarrow F(G_+^{\bullet + 1},X).
\end{equation*}
In our definition of a continuous $\GG$--spectrum, which again is essentially due to Devinatz--Hopkins \cite{DH}, we
replace the functors $ F(G_+^{s + 1},-)$ by the functors $\Func(\GG^{s+1}_+,-)$.

\begin{defn}[{\bf Continuous $\GG$-actions}]\label{def:cont-g-spec}
Let $X$ be a $\Kn$-local spectrum. A \emph{continuous} $\GG$--action on $X$ consists of a map
\[
\eta = \eta_X\colon X \to \Func(\GG_+,X)
\]
so that the diagram
\begin{equation}\label{eq:cosim-g}
X \longr \Func(\GG^{\bullet + 1}_+,X)
\end{equation}
determined by $\eta$ and $\GG$ is an augmented cosimplicial spectrum. 

A map of continuous $\G$--spectra consists of a map of the respective augmented cosimplicial diagrams.
\end{defn} 

\begin{rem}\label{rem:logic} If $X$ has a $\GG$-action, then the composition
\begin{equation}\label{eq:refined}
\xymatrix{
X \ar[r]^-\eta & \Func(\GG_+,X) \ar[r] & F(\GG_+,X)
}
\end{equation}
defines an action (in the usual sense) of $\GG$ on $X$. Conversely given an action of $\GG$ on $X$, we say
that action refines to a continuous action, or simply that the action is continuous, if there is a factoring as in  
\eqref{eq:refined} that gives $X$ the structure of a continuous $\GG$-spectrum. 

This is what Devinatz--Hopkins \cite{DH} accomplish, where the discrete $\G$--action on $\E$ was already given by the 
Goerss--Hopkins--Miller theorem. We discuss this example further in \cref{rem:unpack-DH} below.
\end{rem}

\begin{example}\label{ex:trivial}
A more tautological example is the following: For any $\Kn$-local spectrum $X$, the trivial action of $\G$ on $X$ is continuous. Here we start with $\eta\colon X\to \Func(\G_+, X) $ adjoint to the projection map $\G_+ \smsh X \to X$.
\end{example}

\begin{defn}[{\bf Homotopy fixed points}] 
If $X$ is a continuous $\GG$--spectrum and $\K \subseteq \GG$ is a closed subgroup, we define
$\Func(\G_+,X)^\K = \Func(\GG/\K_+,X)$ and
\begin{align}\label{eq:fixed-pnts}
X^{h\K} &= \holim_\Delta \Func(\GG_+^{\bullet + 1},X)^\K\\
& \simeq \holim_\Delta \Func(\GG/\K_+ \wedge \GG_+^{\bullet},X).\nonumber
\end{align}
\end{defn}

\begin{rem}\label{rem:HFPSS} Suppose further that $\K \subseteq \G$ is a closed subgroup and 
that $X \simeq \holim_i X\smsh \MJ{i}$ is a $\Kn$-local spectrum such that $\pi_t (X \smsh \MJ{i})$ is finite
for all $i$ and $t$. Using \Cref{lem:ContinuousMaps}, one sees that these definitions are designed
so that the Bousfield--Kan spectral sequence associated to \eqref{eq:fixed-pnts} is the homotopy fixed point spectral sequence
\[
E_2^{s,t} \cong H_c^s(\K,\pi_tX) \Longrightarrow \pi_{t-s} X^{h\K}
\]
with $E_2$-term given by the continuous group cohomology. 
\end{rem}

\begin{rem}\label{rem:otherG}
There is an obvious generalization of this definition to other settings, for example the group may be any profinite group. Likewise, the spectrum $X$ may live in another category where analogues of the generalized Moore spectra $\MJ{i}$ play a similar role. For example $X$ may be a $p$-complete spectrum, so $X \simeq \holim_i X \smsh S/p^i$. While we will in effect construct a continuous $p$-complete $\Z_p^\times$ spectrum in this sense in \Cref{sec:TateDet}, we refrain from setting up a general theory.
\end{rem}

The following is an easy but useful property, which we record as a lemma for convenient future reference.
\begin{lem}\label{lem:inclusion-of-fixed}
Let $X$ be a continuous $\G$--spectrum. If $X^{h\G}$ is given the trivial $\G$--action, the ``inclusion of fixed points" map $X^{h\G} \to X$ is $\G$-equivariant.
\end{lem}
\begin{proof}
The map in question is $\holim_\Delta $ of the cosimplicial map 
\[
\Func(\G^{\bullet}_+,X) \simeq \Func(\G^{\bullet+1}_+,X)^{\G} \longrightarrow \Func(\G_+^{\bullet+1},X), 
\]
given by the inclusion of fixed points, which by construction has the required properties.
\end{proof}

One way to summarize the results of Devinatz and Hopkins \cite{DH} is as follows. The phrase ``essentially unique"
means the space of choices is contractible.

\begin{thm}\label{thm:DH-fixed-pts} The $\GG$-spectrum $\E$ has an essentially unique structure as a continuous 
$\GG$-spectrum with the property that if $\K \subseteq \GG$ is closed, then the map of Morava modules
$\E_\ast \E^{h\K} \to \E_\ast \E$ is naturally isomorphic to the inclusion
\[
\Map^c(\GG/\K,\E_\ast) \longr \Map^c(\GG,\E_\ast).
\]
\end{thm}
The Morava modules $\E_\ast \E^{h\K}$ and $\E_\ast \E$ are discussed in more details immediately after \Cref{rem:unpack-DH}.

\begin{rem}\label{rem:unpack-DH} The statement of \Cref{thm:DH-fixed-pts}  at once disguises quite a bit of difficult
work and obscures the logic of the Devinatz--Hopkins argument; thus, it is surely worth going into a bit of detail.

Suppose for a moment that we knew that \Cref{thm:DH-fixed-pts} was true. As above, choose a nested sequence of
open normal subgroups $U_{j+1} \subseteq U_j \subseteq \GG$ with $\cap\ U_j = \{e\}$. Then we would have a sequence
of spectra
\begin{equation}\label{eq:nested-fp}
\cdots \longr \E^{hU_{j}} \longr \E^{hU_{j+1}}  \longr\cdots \longr\E
\end{equation}
with the following properties
\begin{enumerate}

\item $\E^{hU_j}$ is a $\GG/U_j$ spectrum and all the maps of \eqref{eq:nested-fp} are $\GG$-equivariant;

\item the map $\E_\ast \E^{hU_{j}} \longr \E_\ast \E$ of Morava modules is isomorphic to the inclusion
\[
\Map^c(\GG/U_j,\E_\ast) \longr \Map^c(\GG,\E_\ast);
\] 

\item the induced map $\hocolim_j \E^{hU_j} \to \E$ is a $\Kn$-local equivalence.
\end{enumerate}

Let us give some detail on Part (3). By \Cref{rem:compl-tower}, Part (6) we have that if
 $X$ is $L_n$-local then $L_{\Kn}X = \holim X \smsh \MJ{i}$.
The spectra $\E^{hU_j}$ are $\Kn$-local and, hence $L_n$-local. Since $L_n$ is smashing the homotopy colimit is $L_n$-local,
so Part (3) is equivalent to the statement that
\[
\hocolim_j \E^{hU_j} \smsh \MJ{i} \longr \E \smsh \MJ{i}
\]
is an equivalence for all $i$. This follows from (2) and the fact that $\cap\ U_j = \{e\}$. 

Next observe that since $\GG/U_j$ is finite, $\E_\ast \E^{hU_{j}}$ is finitely generated as an $\E_\ast$-
module, hence $\E^{hU_{j}}$ is dualizable in the $\Kn$-local category, by \cite[Thm. 8.6]{hovstrmemoir}.
Putting all this together -- and still
assuming we know \Cref{thm:DH-fixed-pts} --  we would have
the following diagram of cosimplicial spectra, with the vertical maps being $\Kn$-local equivalences
\begin{equation}\label{eq:define-sq}
\xymatrix{
\hocolim_j \E^{hU_j} \rto \dto_\simeq & \hocolim_j F((\GG/U_j)_+^{\bullet+1},\E^{hU_j}) \dto^\simeq \\
\E \rto &\Func(\GG_+^{\bullet+1},\E).
}
\end{equation}

Devinatz and Hopkins prove \Cref{thm:DH-fixed-pts} by reversing the logical order of this discussion: Recall that the Goerss--Hopkins--Miller theorem provides $\E$ with an essentially unique structure as an $E_\infty$-ring spectrum,
the space $\map_{E_\infty}(\E,\E)$ has contractible components, and $\pi_0\map_{E_\infty}(\E,\E)\cong \GG$. This
gives $\E$ an essentially unique structure as $\GG$-spectrum, with the action through $E_\infty$-ring maps. 
 
Using the Goerss--Hopkins--Miller Theorem, Devinatz and Hopkins define a sequence of spectra which they call $\E^{hU_j}$ and
maps as in \eqref{eq:nested-fp} satisfying Parts (1)--(3) above. They then define the continuous $\GG$-structure on $\E$
using the diagram of \eqref{eq:define-sq}. Then they must justify the notation $\E^{hU_j}$; that is, they must show
the spectra defined this way agree, up to equivalence, with the fixed points as defined in \eqref{eq:fixed-pnts}.
Finally,  they must calculate $\E_\ast \E^{h\K}$. For this they use the remarkable \Cref{prop:DHSS} below.
\end{rem}

We further unpack the statement of \Cref{thm:DH-fixed-pts} and generalize it (\Cref{prop:CommuteHFP} and \Cref{cor:ActMoravaModule}). 
For any $X$, 
\[
\E_\ast (\E \smsh X) = \pi_\ast L_{\Kn}(\E \smsh \E \smsh X)
\]
is a Morava module, using the action of $\GG$ on the left factor
$\E$. Now, suppose $X$ itself has a $\GG$-action so that the diagonal action on $\E \smsh X$ is continuous. If $h \in \GG$ and $x \in \E_\ast X$, then we write $h\ast_d x$ for this action. The adjoint of the diagonal action of $\GG$ on $\E \smsh X$ gives rise to a map
\begin{equation}\label{eq:fund-iso}
\eta\colon \E_\ast (\E \smsh X) \longr \Map^c(\GG,\E_\ast X).
\end{equation}
Explicitly, if $x\colon S^t \to \E \smsh \E \smsh X$ and $g \in \GG$, then $\eta_x(g)$ is the composite
\[
\xymatrix{
S^t \rto^-x & \E \smsh \E \smsh X \rto^-{1 \smsh g \smsh g} & \E \smsh \E \smsh X \rto^-{\mu \smsh 1} & \E \smsh X,
}
\]
where $\mu$ is multiplication and we have suppressed the $\Kn$-localizations.

If $X$ is $S^0$ with the trivial action, then $\eta$ gives the identification of $\E_*\E$ with $\Map^c(\G, \E_*)$ which appeared in \Cref{thm:DH-fixed-pts}. The following result covers every case that arises in this note. 
\begin{lem}\label{lem:etaiso} Suppose $X = Y \wedge Z$ where $\Kn_\ast Y$ is zero in odd degrees and $Z$ is a $\Kn$-locally dualizable spectrum. Then the map $\eta$ in \eqref{eq:fund-iso} is an isomorphism. 
\end{lem}

\begin{proof}
As in the proof of \cite[Prop.~2.4]{ghmr}, it suffices to show that the natural map
\[
\E_*(\E \smsh X) \longrightarrow \lim_i\E_*(\E \smsh \MJ{i} \smsh X)
\]
occurring in the Milnor sequence associated to $\holim_i\E \smsh \E \smsh X \smsh \MJ{i}$ is an isomorphism, i.e., that $\lim_i^1\E_{*}(\E \smsh \MJ{i} \smsh X) = 0$. The assumption on $Y$ implies that $\E_*(Y)$ is a flat $\E_*$-module, so there is an isomorphism
\[
\E_{*}(\E \smsh \MJ{i} \smsh X) \cong \E_*\E\otimes_{\E_*} \E_*(Y) \otimes_{\E_*} \E_*(\MJ{i} \smsh Z).
\]
Since $Z$ is dualizable, $\MJ{i} \smsh Z$ is $\Kn$-locally compact, hence $\E_*(\MJ{i} \smsh Z)$ is finite. This shows that the tower $(\E_*(\MJ{i} \smsh Z))_i$ is Mittag-Leffler, which implies that the required $\lim^1$ vanishes. 
\end{proof}

\begin{rem}\label{rem:all-the-actions} We now have (at least) two actions to keep straight.
\begin{enumerate}
\item For the Morava module structure on  $\E_\ast (\E \smsh X)$ 
the isomorphism $\eta$ becomes $\GG$-equivariant if we give the
module of functions the conjugation action
\[
(h\phi)(g) = h\ast_d \phi(h^{-1}g).
\]

\item The diagonal action on $\E \smsh X$ gives an action of $\GG$ on $\E_\ast (\E \smsh X)$; this involves
the right factor of $\E$. With respect to this action $\eta$ becomes $\GG$-equivariant if we give the
module of functions the action
\[
(h\star \phi)(g) = \phi(gh).
\]
\end{enumerate}
Note that the two actions commute. 
\end{rem}

At this point, we need the following remarkable result due to Devinatz and Hopkins. 

\begin{prop}\label{prop:DHSS} Let $W^\bullet$ be a cosimplicial spectrum. Suppose there exists 
an integer $N$ and a finite type $0$ spectrum $Y$ so that for all spectra $Z$  the Bousfield--Kan spectral sequence
\[
\pi^s \pi_t F(Z, Y\smsh W^\bullet ) \Longrightarrow \pi_{t-s}F(Z, \holim_{\Delta} (Y \smsh W^\bullet) )
\]
has a horizontal vanishing line of intercept $s=N$ at the $E_\infty$-page. Then for any spectra $A $ and $F$ and maps $v\colon \Sigma^k A\to A$,
there is an equivalence
\[
v^{-1}L_F(A \smsh \holim_{\Delta} W^\bullet) \simeq \holim_{\Delta} (v^{-1} L_F (A \smsh W^\bullet)). 
\]
\end{prop}

\begin{proof} This is all contained in \cite[\S 5]{DH}, even if it is not explicitly stated this way. More specifically,
we combine the material before their Lemma 5.11, Lemma 5.12, and the argument given in the proof of their 
Theorem 5.3, substituting our $Y$ for their spectrum $X$.
\end{proof}

\begin{prop}\label{prop:CommuteHFP}
Let $X$ be a $\G$-spectrum, which is ($\Kn$-locally) dualizable, and such that the diagonal action of $\G$ on $\E\smsh X$ is continuous. 
Then for a closed subgroup $\K$ of $\G$ and any spectrum $A$,
there is a $\Kn$-local equivalence 
\[ 
A \smsh (\E \smsh X)^{h\K} \simeq (A\smsh \E \smsh X)^{h\K}, 
\]
where on the right-hand side, $\K$ is acting trivially on the first factor. 
\end{prop} 

\begin{proof}
We will prove this by applying \Cref{prop:DHSS} (with $F = \Kn$, and $v=\mathrm{id}$) to the
cosimplicial spectrum which computes the homotopy fixed points $(\E \smsh X)^{h\K}$.
Specifically, $(\E \smsh X)^{h\K} \simeq \holim_{\Delta} W^\bullet$, with 
\[
W^s = \Func(\G_+^{s+1}, \E\smsh X )^{\K} = \Func(\G/\K_+ \smsh \G^s_+, \E\smsh X).
\]
We need to check that the conditions of \Cref{prop:DHSS} are satisfied; then the result follows. The argument
we give exactly mirrors that of \cite[Theorem 5.3]{DH}.

We choose $Y$ to be a finite type 0 spectrum so that $\E_0Y$ is free as a $C$-module for every cyclic subgroup $C \subseteq \G$ of
order $p$ and so that $\E_1 Y = 0$. Moreover, $\E_* Y$ is free as an $\E_*$-module. Such a spectrum $Y$ is constructed by Jeff Smith; see \cite[\S 6.4, 8.3, 8.4]{RavNil}.

Since both $X$ and $Y$ are dualizable, \Cref{lem:ContinuousMaps} gives us that, for any spectrum $Z$, there is an isomorphism
\[
\pi_t F(Z,Y\smsh W^s) \cong   \Map^c(\G^{s+1}, \pi_tF(Z, \E\smsh X \smsh Y))^\K.
\]
Using again that $X$ and $Y$ are dualizable as well as that $\E_* Y$ is in even degrees and free over $\E_*$,
\[
\pi_t F(Z,\E\smsh X \smsh Y) \cong \E^{-t}(Z\smsh DX) \otimes_{\E_0} \E_0(Y).
\]
Now $\E_0(Y)$ is free as a $C$-module for every cyclic subgroup $C \subseteq \G$ of order $p$,
so the same is true for $\pi_t F(Z,\E\smsh X \smsh Y)$, and that fact implies that 
\[
\pi^s \pi_t F(Z, Y\smsh W^\bullet) \cong H^s(\K, \E^{-t}(Z\smsh D X) \otimes_{\E_0} \E_0(Y))
\]
is zero for $s>n^2$ \cite[Lemma 8.3.5]{RavNil}.\footnote{The quoted results only claims the vanishing for $s > N$ where 
$N$ depends only on $n$ and $p$. To get $N=n^2$ would require reworking the proof and using that $\GG$ has
virtual Poincar\'e duality of dimension $n^2$.} In particular, this gives a horizontal vanishing line at the $E_2$-page, and 
\Cref{prop:DHSS}  applies to give the claim.
\end{proof}

\begin{cor}\label{cor:ActMoravaModule}
If $X$ and $\K$ are as in \Cref{prop:CommuteHFP}, then there is an isomorphism of Morava modules
\[\E_\ast((\E \smsh X)^{h\K})\cong \Map^c(\GG/\K,\E_\ast X)\]
where the Morava module structure on the right-hand side is the conjugation action described in \Cref{rem:all-the-actions}.
\end{cor}

\begin{proof} \Cref{prop:CommuteHFP} implies that
$\E_\ast((\E \smsh X)^{h\K}) \cong \pi_*(\E \smsh \E\smsh X)^{h\K}$. We will use the homotopy fixed point spectral 
sequence computing $\pi_*(\E \smsh \E\smsh X)^{h\K}$. As was discussed in \Cref{rem:all-the-actions}, there is a $\K$-
equivariant isomorphism \[\E_*( \E\smsh X) \cong \Map^c(\G, \E_*X)\] with the $\K$-action on $\E_*(\E\smsh X) = \pi_*(\E 
\smsh \E \smsh X)$ the diagonal action on the right two factors and the $\K$-action on $\Map^c(\G, \E_*X)$ is right 
multiplication on the source. It follows that the $E_2$-term of the homotopy fixed point spectral sequence is 
\[
H^*(\K, \pi_*(\E \smsh \E\smsh X)) \cong H^*(\K, \Map^c(\G, \E_*X)).
\] 
Furthermore
\[
H^s(\K, \Map^c(\G, \E_*X)) \cong
\begin{cases}
\Map^c(\GG/\K,\E_\ast X), &s=0;\\
0,&s\ne 0
\end{cases}
\]
since $\Map^c(\G, \E_*X)$, is induced as $\GG$-module, and hence as $\K$-module. Thus, the
homotopy fixed point spectral sequence collapses and the edge homomorphism gives an isomorphism
of Morava modules
\begin{align*}
\E_\ast((\E \smsh X)^{h\K})&  \xra{\cong} (\E_\ast(\E \smsh X))^\K \cong \Map^c(\GG/\K,\E_\ast X).\qedhere
\end{align*}
\end{proof}

\section{The Tate sphere and the determinant sphere}\label{sec:TateDet}

In order to define the determinant sphere, we need a spectrum-level construction which twists actions. This is accomplished by 
a sphere spectrum we suggestively denote by $\ST$, to be indicative of a Tate twist. Namely, $\ST$ is the $p$-completed 
sphere spectrum $S^0$ with a continuous action of $\Z_p^{\times}$ coming from its action as automorphisms on
$\pi_0S^0$, to be constructed below.

We can also consider $\ST$ as a spectrum with a $\G$-action, where $\G$ acts through the determinant homomorphism
\[
\det \colon \G \longrightarrow \Z_p^{\times},
\]
defined as in \cite[Section 1.3]{ghmr}. The determinant is a surjection and we let $\SG$ denote its kernel, so that there is an exact sequence
\[
1 \longrightarrow \SG \longrightarrow \G \longrightarrow \Z_p^{\times} \longrightarrow 1. 
\]
We will then define $\Sdet$ as the homotopy fixed points of a particular $\GG$-spectrum in the $\Kn$-local category.

We now begin the construction of $\ST$; we will start by constructing a discrete action of a dense subgroup of $\Z_p^\times$. If $p > 2$, we have a decomposition
\[
(1 +p\ZZ_p) \times \mmu \cong \ZZ_p^\times
\]
where $\mmu=\FF_p^\times$ is the cyclic group of order $p-1$ given by the Teichm\"uller lifts. Let $\Cc \subseteq 1+p\ZZ_p$
be the infinite cyclic subgroup generated by $\tau = 1+p \in 1 +p\ZZ_p$. 

If $p=2$, we have a slightly different decomposition
\[
(1 +4\ZZ_2) \times \mmu \cong \ZZ_2^\times
\]
where now $\mmu = \{ \pm 1 \}$. Let $\Cc$ be generated by $\tau = 1+4 =5 \in 1+4\ZZ_2$.

With this setup, we write $G=\Cc \times \mmu$ for all primes. Note that $G$ is a dense subgroup of $\Z_p^\times$, and $\tau$ is a generator of the torsion-free subgroup $\Cc\cong \ZZ$.
If $p> 2$ the inclusion $\Cc \to 1+p\ZZ_p$ completes to an isomorphism $\ZZ_p \cong 1+p\ZZ_p$.
At $p=2$ we get a similar isomorphism $\ZZ_2 \cong 1 +4\ZZ_2$.

\begin{prop}\label{prop:the-disc-action} The inclusion $G \to \ZZ_p^\times = \pi_1B\haut(S^0)$ can be
canonically realized by a map
\[
BG \longr B\haut(S^0).
\]
\end{prop}

\begin{proof} Since $B\haut(S^0)$ is an infinite loop space we need only realize separately the maps $\Cc \to \ZZ_p^\times$
and $\mmu \to \ZZ_p^\times$ as maps $B\Cc \to B\haut(S^0)$ and $B\mmu \to B\haut(S^0)$. The map we want
will then be the composite
\[
BG \simeq B\Cc \times B\mmu \longrightarrow B\haut(S^0) \times B\haut(S^0) \longrightarrow B\haut(S^0)
\]
where the second map is the loop space multiplication. 

At all primes $B\Cc \simeq B\ZZ \simeq S^1$ and the choice of $\tau$ defines the required map $S^1 \to B\haut(S^0)$.

If $p=2$, then $B\mmu \simeq B\ZZ/2 \simeq BO(1)$ and the map we need is defined by the composition
\[
BO(1) \longrightarrow BO \longrightarrow B\haut(S^0).
\]
Suppose $p>2$ and let $A$ be some $2$-skeleton of $B\mmu$. The inclusion $\mmu \subseteq \ZZ_p^\times$ defines a map
$A \to B\haut(S^0)$ by extending a generator of $\mmu \subset \pi_1B\haut(S^0)$ to $A$. Since $\pi_i B\haut(S^0) \cong \pi_{i-1}S^0$ is $p$-complete for $i \geq 2$ and $\mmu$ has order prime to $p$, the map out of $A$ extends uniquely to a map $B\mmu \to B\haut(S^0)$.
\end{proof}

Let $k \geq 1$ and let $G_k \subseteq G$ be the kernel of the composition
\[
\xymatrix{
G \rto^-{\subseteq} & \ZZ_p^\times \rto & (\ZZ/p^k)^\times.
}
\]
If $p > 2$, then $G_1 =\Cc$ and $G_k$ is infinite cyclic generated by  $\tau^{p^{k-1}}$. If $p = 2$,
then $G_2 = \Cc$ and for $k > 1$ the group $G_k$ is infinite cyclic generated by $\tau^{p^{k-2}}$. We have that the intersection $\cap G_k $ is trivial, and $\lim_k G/G_k \cong \ZZ_p^\times$; thus, the subgroups $G_k$ define the usual topology on $\ZZ_p^\times$. 

\begin{prop}\label{prop:moore-tower} Let $\STa$ be the $p$-complete sphere spectrum with the discrete action of $G$ constructed above. If $p$ is odd let $k\geq 1$ and if $p =2$
let $k > 1$. Then there is
an equivalence
\[
S/p^k \simeq EG_+ \wedge_{G_k} \STa
\]
and the residual action of $G/G_k \cong (\ZZ/p^k)^\times$ realizes the standard action of $(\ZZ/p^k)^\times$ on
$\ZZ/p^k \cong \pi_0S/p^k$.
\end{prop}

\begin{proof} The homotopy orbit spectrum $EG_+ \wedge_{G_k} \STa$ is a connected spectrum and we have
a homotopy orbit spectral sequence for $H_\ast(-) = H_\ast(-,\ZZ)$:
\[
E^2_{p,q} \cong H_p(G_k,H_q\STa) \Longrightarrow H_{p+q} (EG_+ \wedge_{G_k} \STa).
\]
Let $p > 2$. The group $G_k$ is infinite cyclic generated by $\tau^{p^{k-1}}$ where $\tau = 1+p$.
Since $\tau^{p^{k-1}} \equiv 1+p^k$
modulo $p^{k+1}$ we have $E^2_{p,q} = 0$ unless $(p,q) = (0,0)$ and there is a surjection of $G$-modules
\[
\ZZ_p \cong H_0(\STa) \longr H_0(G_k,H_0(\STa)) \cong \ZZ/p^k.
\]
It follows that $EG_+ \wedge_{G_k} \STa$ must be a Moore spectrum for $\ZZ/p^k$ with the standard action of $\ZZ/p^k$ on $\pi_0S/p^k$. The proof at the prime $2$ is completely analogous.
\end{proof}

Recall that continuous actions were discussed in \Cref{sec:continuousactions}. See in particular \Cref{def:cont-g-spec} and \Cref{rem:otherG}.
\begin{prop}\label{prop:cont-zz-Sone} The $G$-action on $\STa$ extends to a continuous action of the
profinite group $\ZZ_p^\times$, in the sense that we have an augmented cosimplicial spectrum
\[
\STa \longrightarrow \Func((\ZZ_p^\times)_+^{\bullet +1},\STa), 
\]
so that the augmentation refines the $\Z_p^\times$--action.
\end{prop}
\begin{proof} Write $S/p^k(1)$ for $EG_+ \wedge_{G_k} \STa$ with its $G/G_k \cong (\ZZ/p^k)^\times$-action. Then the
augmented cosimplicial spectra
\[
S/p^k(1) \to F((G/G_k)_+^{\bullet +1},S/p^k(1))
\]
assemble to give a map
\begin{align*}
\STa \simeq \holim_k S/p^k(1) \longrightarrow & \holim_k \hocolim_j F((G/G_j)_+^{\bullet +1},S/p^k(1)) \\
& = \Func((\ZZ_p^\times)_+^{\bullet +1},\STa)
\end{align*}
as needed. 
\end{proof}

\begin{defn}\label{defn:S(1)} We will write $\ST$ for the $p$-complete sphere $S^0$ with the continuous
$\ZZ_p^\times$-action of \cref{prop:cont-zz-Sone}. 
The same construction gives $\ST$ as a continuous $p$-complete $\GG$--spectrum, where $\GG$ acts through the determinant surjection $\det\colon\GG\to\Z_p^\times$.

We refer to this equivariant sphere as the \emph{Tate sphere}.
\end{defn}

Now we take the Morava $E$-theory spectrum $\E$ and give $\E \smsh \ST$ the
diagonal $\GG$-action. The next result indicates that this is an interesting construction. 

\begin{prop}\label{prop:homotopy-ES} There is an isomorphism of Morava modules
\[
\E_\ast\bdet \cong \pi_\ast (\E \smsh \ST) = \E_\ast \ST.
\]
\end{prop}
\begin{proof}
The edge map of the Tor spectral sequence
\[ 
\E_\ast\bdet = \E_\ast \otimes_{\pi_0S^0} \pi_0 \ST \longrightarrow \pi_\ast(\E \smsh \ST) 
\]
is an isomorphism, and respects the $\G$-action by the naturality of the spectral sequence.
\end{proof}

The following technical result is the key to our calculations.

\begin{prop}\label{cont-for-ES1}  The $\GG$-spectrum $\E \smsh \ST$ has the structure of a $\Kn$-local
continuous $\GG$-spectrum.
\end{prop}

\begin{proof} As in \eqref{eq:cosim-g} we need to construct an augmented cosimplicial $\GG$-spectrum
\[
\E \smsh \ST \longrightarrow \Func(\GG_+^{\bullet + 1},\E \smsh \ST)
\]
so that the augmentation refines the $\GG$-action on $\E \smsh \ST$. 

As above, we continue writing $S/p^k(1)$ for $EG_+ \wedge_{G_k} \STa$ with its $G/G_k \cong (\ZZ/p^k)^\times$ action.
Let us also write $S/p^k$ for the Moore spectrum when we do not need to refer to  the action. 

Since $\MJ{i}$ and $S/p^k$ are finite spectra we have
\begin{align*}
\Func(\GG_+^s,\E \smsh \ST) &= \holim_i \hocolim_j F((\GG/U_j)^s_+,\E \smsh \ST \smsh \MJ{i})\\
& \xrightarrow{\simeq} \holim_k \holim_i \hocolim_j F((\GG/U_j)^s_+,\E \smsh S/p^k(1) \smsh \MJ{i});
\end{align*}
indeed, both sides of the last equivalence are $p$-complete and the natural map between them is an equivalence after smashing with $S/p$. For all $j$ so that $U_j$ is in the kernel of
\[
\xymatrix{
\GG \rto^-\det & \ZZ_p^\times \rto & (\ZZ/p^k)^\times,
}
\]
the diagonal action of $\GG/U_j$ on $\E^{hU_j} \smsh S/p^k(1) \smsh \MJ{i}$ defines an augmented
cosimplicial $\GG$-spectrum
\[
\E^{hU_j} \smsh S/p^k(1) \smsh \MJ{i} \longr F((\GG/U_j)^{\bullet +1}_+,\E^{hU_j} \smsh S/p^k(1) \smsh \MJ{i}).
\]
Since $\hocolim_j \E^{hU_j} \simeq \E$, these assemble into the cosimplicial spectrum we need. 
\end{proof}

We can now make our central definition.

\begin{defn}\label{def-Sdet} The determinant sphere is the spectrum
\[
\Sdet = (\E \smsh \ST)^{h\GG} = \holim_\Delta \Func(\GG_+^{\bullet +1},\E\smsh \ST)^\GG.
\]
\end{defn}

\begin{rem}\label{rem:some-ss} If $\K \subseteq \G$ is closed we defined (\Cref{def:cont-g-spec})
\[
(\E \smsh \ST)^{h\K} = \holim_\Delta \Func(\G_+^{\bullet+1},\E \smsh \ST)^\K.
\]
Therefore, using \Cref{prop:homotopy-ES} and \Cref{rem:HFPSS}, we have a homotopy fixed point spectral sequence
\[
H^s_c(\K,\E_\ast \bdet) \Longrightarrow \pi_{t-s}(\E\smsh \ST)^{h\K}.
\] 
\end{rem}

We now must show that there is an isomorphism of Morava modules $\E_\ast \Sdet \cong \E_\ast \bdet$. 
But this follows directly from \Cref{prop:homotopy-ES} and 
\Cref{cor:ActMoravaModule}.

\begin{prop}\label{prop:what-were-here-for} There is an isomorphism of
Morava modules
\[
\E_\ast \Sdet \cong \E_\ast\DDet.
\]
\end{prop}

We now extend this map to an equivalence of spectra. Let $\iota\colon\Sdet = (\E\smsh \ST)^{h\GG} \to \E \wedge \ST$
be the inclusion of the fixed points from \cref{lem:inclusion-of-fixed}, and let $\mu\colon \E \wedge \E \to \E$ be the multiplication. Define 
\[
f\colon \E \smsh \Sdet \longrightarrow \E \smsh \ST
\]
to be the composition
\begin{equation}\label{eq:keycomposite}
\xymatrix{
\E \smsh \Sdet \ar[r]^-{1 \smsh \iota} & \E \smsh \E \smsh \ST \ar[r]^-{\mu \smsh 1} & \E \smsh \ST.
}
\end{equation}
This map is $\GG$-equivariant if we use the action on $\E$ on the source and the diagonal action on the target. 

\begin{thm}\label{thm:pretty-damn-cool} The map $f\colon \E \smsh \Sdet \to \E \smsh \ST$
of \eqref{eq:keycomposite} is a $\GG$-equivariant equivalence and induces the isomorphism of Morava modules
\[
\E_*\Sdet \cong \E_*{\left<\det \right>}.
\]
of \Cref{prop:what-were-here-for}.
\end{thm}

\begin{proof} To check that $f$ is an equivalence we need only check that it induces the indicated map on Morava modules. 
Applying $\pi_*(-)$ to \eqref{eq:keycomposite} gives
\begin{align}\label{eq:compo}
\xymatrix{
\E_*\Sdet \ar[r] \ar[d]^-{\cong}& \E_*( \E \smsh \ST) \ar[d]^-{\cong} \ar[r]^-{\mu \smsh 1} & \E_* \ST \ar[d]^-{=}  \\
\Map^c(\G, \E_*\ST )^{\G} \ar[r] & \Map^c(\G, \E_*\ST) \ar[r]  & \E_*\ST.}
\end{align}
The first vertical isomorphism is from \Cref{cor:ActMoravaModule}, whereas the second is the isomorphism of \Cref{lem:etaiso}.
In the bottom row, the first map is the inclusion of fixed points and the second map is evaluation at the unit $e \in \G$. 
The fixed points on the bottom left are exactly the constant functions, so the composite is an isomorphism as claimed.
\end{proof}

This yields the following practical invariance result. 

\begin{cor}\label{cor:invariance}
If $\K $ is a closed subgroup of $\G$ which is in the kernel of the determinant, then $\E^{h\K} \smsh \Sdet \simeq \E^{h\K}$.
\end{cor}

\begin{proof} We use \Cref{thm:pretty-damn-cool}. 
When we restrict the $\G$-action on the Tate sphere $\ST$ to $\K$, we get that $\K$ acts trivially, so $\ST$ is $\K$-equivariantly equivalent to $S^0$.  We have
\[ \E^{h\K} \smsh \Sdet \simeq (\E \smsh \Sdet)^{h\K} \simeq (\E \smsh \ST)^{h\K} \simeq \E^{h\K}, \]
where the first equivalence follows since $\Sdet$ is a $\Kn$-locally dualizable spectrum with trivial $\K$-action.
\end{proof}

\begin{rem} 
The specifics of the determinant homomorphism are not relevant for this construction and its immediate properties. Indeed, for any continuous homomorphism $\phi \colon \G \to \Z_p^\times$, we may define a $\Kn$-local $\phi$-twisted sphere by the formula
\[S\langle \phi \rangle = (\E \smsh \ST)^{h\G}, \]
where on the right hand side $\G$ acts diagonally, and through $\phi$ on $\ST$. The proof of \Cref{prop:what-were-here-for} generalizes to compute the corresponding Morava module as 
\[
\E_\ast S \langle \phi \rangle \cong \E_\ast\langle \phi\rangle,
\]
where the right hand side denotes the action on $\E_\ast$ obtained by twisting the standard action with $\phi$.
This construction amounts to giving the dashed lift as indicated in the following diagram involving the group $\Pic_n^0$ of $\Kn$-local spectra $X$ with $\E_*X \cong \E_*$ and the algebraic Picard group $(\Pic_n)_{\alg}^0$ of invertible $\G$-$\E_0$-modules:
\[
\xymatrix{& \Pic_n^0 \ar[d]\\
H_c^1(\G,\ZZ_p^\times) \ar[r] \ar@{-->}[ru] & (\Pic_n)_{\alg}^0 \cong H^1_c(\G, \E_0^{\times}).}
\]
The bottom horizontal map is induced by the inclusion $\ZZ_p^{\times} \to \E_0^{\times}$.

We note that the determinant homomorphism topologically generates most of the image of the depicted horizontal arrow, so we are not losing much information by restricting our attention to its study.  In particular, Westerland's version of the determinant \cite{craig_imj} and ours have the same image in the algebraic Picard group. Indeed, they agree on $\mathbb{S} \subseteq \G$ and the map from $H_c^1(\G,\ZZ_p^\times)$ to $(\Pic_n)_{\alg}^0$ factors through
\[H^1_c(\G, \W^{\times}) \cong H^1_c(\mathbb{S}, \W^{\times})^{\Gal}. \]
\end{rem}

\section{Deconstructing the determinant sphere}

Let $S\G \subseteq \GG$ be the kernel of the determinant. Then we can form the fixed point spectrum $\E^{h \SG}$.
This will have a residual action of $\GG/\SG \cong \ZZ_p^\times$. (See the paragraph before Theorem 4 in \cite{DH}.) Furthermore
\[
(\E \wedge \ST)^{h \SG} \simeq \E^{h \SG} \wedge \ST,
\]
where the right hand side has a diagonal $\ZZ_p^\times$-action. 

At odd primes we get a simple description of $\Sdet$ directly from Devinatz--Hopkins fixed point theory. 

\begin{prop}\label{prop:fiber1} Let $p > 2$ and let $\phi \in \GG$ be any element so that $\det(\phi)$ topologically
generates $\ZZ_p^\times$. Then there is a fiber sequence
\[
\xymatrix{
\Sdet \rto & \E^{h \SG} \ar[rr]^-{\det(\phi)\phi-1} && \E^{h \SG}.
}
\]
\end{prop}

\begin{proof} By construction, the action of $g \in \GG$ on $S(1)$ is given, up to homotopy, by multiplication by $\det(g) \in \ZZ_p^\times$. Thus, the diagonal action of $\phi$ on $\E^{h\SG} \wedge S(1)$ is, up to homotopy,
given by 
\[
\phi \wedge \det(\phi)\colon \E^{h \SG} \wedge \ST \longr \E^{h \SG} \wedge \ST.
\]
Using that $\ST$ is non-equivariantly the sphere $S^0$, we have  a 
homotopy commutative diagram
\[
\xymatrix{
\E^{h \SG} \wedge \ST \ar[rr]^-{\phi \wedge \det(\phi) - 1} \dto_\simeq && \E^{h \SG} \wedge \ST\dto^\simeq\\
\E^{h \SG} \ar[rr]^-{\det(\phi)\phi-1} && \E^{h \SG}.
}
\]
Let $F$ be fiber of the bottom map. The composition
\[
\xymatrix{
\Sdet = (\E \wedge \ST)^{h\GG} \rto & \E^{h \SG} \wedge \ST \ar[rr]^-{\phi \wedge \det(\phi) - 1} && \E^{h \SG} \wedge \ST
}
\]
is null-homotopic, so we get a map $f\colon\Sdet \to F$. Using the fact that
\[
\E_\ast \E^{h \SG} \cong \Map^c(\GG/\SG,\E_\ast) \cong  \Map^c(\ZZ_p^\times,\E_\ast)
\]
we compute that $f$ induces an isomorphism of Morava modules.
\end{proof}

We can refine the fiber sequence of \Cref{prop:fiber1}. We still have $p > 2$ and we have a splitting
\[
\mmu \times (1+p\ZZ_p) \cong \ZZ_p^\times.
\]
The group $\mmu \cong \FF_p^\times$ is cyclic of order $p-1$ and  $(1+p\ZZ_p)$ is isomorphic to $\ZZ_p$ itself.

Let $\alpha \in \W^\times \subseteq \GG$ be a primitve $(p^n-1)$st
root of unity; then $\det(\alpha) \in \mmu$ is a generator. The group $\mmu \subseteq \ZZ_p^\times$ acts
on $\E^{h \SG}$ and, since this group is abstractly isomorphic to $C_{p-1}$, the spectrum $\E^{h \SG}$ splits as a
wedge of the
eingenspectra for this action. Let $\E_\chi^{h \SG}$ be the summand defined by the equations
\[
\pi_\ast \E_\chi^{h \SG} = \{\ x \in \pi_\ast \E^{h S \GG}\ |\ \alpha_\ast x = \det(\alpha)^{-1} x\ \}.
\]
Note that the spectrum $\E_\chi^{h \SG} $ corresponds to $(\E^{h\SG}\smsh S(1))^{h\mmu}$.
Indeed, forgetting the $\mmu$-action and remembering that the underlying spectrum of $S(1)$ is the $p$-complete sphere, the map which sends $x\in \pi_*(\E^{h\SG})$ to $x\smsh 1 \in \pi_*(\E^{h\SG}\smsh S(1))$ is a non-equivariant isomorphism.
Now note that if $\alpha_*(x) =  \det(\alpha)^{-1} x$ in $\pi_*\E^{h\SG}$ then $\alpha_*(x \smsh 1) = \alpha_*(x) \smsh \det(\alpha) = x \smsh 1$ in $\pi_*(\E^{h\SG}\smsh S(1))$ so that 
\[x\smsh 1 \in (\pi_*(\E^{h\SG}\smsh S(1)))^{\mmu} \cong \pi_*(\E^{h\SG}\smsh S(1))^{h\mmu}.\]

\begin{prop}\label{prop:fiber2} Let $p > 2$ and let $\psi \in \GG$ be any element so that $\det(\psi)$ topologically
generates $1+p\ZZ_p \subseteq \ZZ_p^\times$. Then there is a fiber sequence
\[
\xymatrix{
\Sdet \rto & \E_\chi^{h \SG} \ar[rr]^-{\det(\psi)\psi-1} && \E_\chi^{h \SG}.
}
\]
\end{prop}

The proof is very similar to  that of \Cref{prop:fiber1}. This fiber sequence appears in \cite[Rem.~2.5]{GHM_pic}
although there is a typo there: the factor of $\det(\psi^{p+1})$ should be replaced by $\det(\psi)^{-(p+1)}$ in Equation (2.6).

At the prime $2$ we have $\ZZ_2^\times \cong \mmu \times (1+4\ZZ_p)$ for $\mmu =\{\pm 1\}$ and the decomposition is more subtle. In particular, $\E^{h \SG}$ does not decompose as a wedge of $\mmu$-eigenspectra, where $\mmu$ acts 
on $\E^{h \SG}$ through $\ZZ_2^{\times} \cong \G/\SG$. Thus we need a replacement. The
following construction expands on ideas of Hans-Werner Henn. 

It follows from its construction in \Cref{prop:the-disc-action} that, as a $\mmu$-spectrum, $S(1)$ is $S^{\sigma -1}$ where $
\sigma$ is the one-dimensional sign representation of $\mmu$. We have a fiber sequence of $\mmu$-spectra
\[
S^{\sigma-1} \to S^0 \wedge \mmu_+ \to S^0
\]
where $S^0$ has the trivial action. This is a fiber sequence of $\ZZ_2^\times$-spectra by
restriction along the quotient map $\ZZ_2^\times \to \mmu$ with kernel $1 + 4\ZZ_2$. 

We smash this sequence with $\E^{hS\GG}$ and use the diagonal action
to obtain a fiber sequence of $\ZZ_2^\times$-spectra
\[
\E^{hS\GG} \wedge S^{\sigma-1} \to \E^{hS\GG} \wedge \mmu_+ \to \E^{hS\GG}.
\]
Now take
$\mmu$-homotopy fixed points to get a fiber sequence of $\ZZ_2 \cong \ZZ_2^\times/\mmu$-spectra. We give a special name to the fiber, i.e., we denote by $\E^{h \SG}_{-} $ the spectrum
\[
\E^{h \SG}_{-} = (\E^{h\SG} \wedge S^{\sigma-1})^{h\mmu},
\]
where $\mmu$ acts diagonally on the right-hand side.  Thus, we have the fiber sequence
\begin{equation}\label{eq:the-funny-cofibering}
\E^{h \SG}_{-} \longr \E^{h \SG}   \xrightarrow{ \ \mathrm{tr} \ } (\E^{h \SG})^{h\mmu},
\end{equation}
where $\mathrm{tr}$ is the transfer.

Now let $\psi \in \GG$ is any element so that $\det(\psi)$ topologically
generates $1+4\ZZ_2$. Since \eqref{eq:the-funny-cofibering} is a cofiber sequence
of $ \ZZ_2^\times/\mmu$-spectra there is an extension of the map
$\psi\colon\E^{h \SG} \to \E^{h \SG}$ to a commutative diagram
\[
\xymatrix{
\E^{h \SG}_{-}\dto_\psi  \rto & \E^{h \SG}\dto^\psi \\
 \E_{-}^{h \SG} \rto &\E^{h \SG}.
}
\]

\begin{prop}\label{prop:fiber3} Let $p = 2$. Then there is a fiber sequence
\[
\xymatrix{
\Sdet \rto & \E_{-}^{h \SG} \ar[rr]^-{\det(\psi)\psi-1} && \E_{-}^{h \SG}.
}
\]
\end{prop}

\begin{proof} The argument is essentially the same as in \Cref{prop:fiber1}. Here is more detail.  By construction
\[
\Sdet = (\E \wedge S(1))^{h\GG} \simeq (\E^{h\SG} \wedge S(1))^{h\ZZ_2^\times}. 
\]
Using the decomposition $\ZZ_2^\times = (1+4\ZZ_2 )\times \mmu$ we obtain a fiber sequence
\[
\xymatrix@C=40pt{
\Sdet \ar[r] & (\E^{h\SG} \wedge S(1))^{h\mmu} \ar[rr]^-{(\psi \wedge \det(\psi)-1)^{h\mmu}} && 
(\E^{h\SG} \wedge S(1))^{h\mmu}.
}
\]
Here we are again using that, up to homotopy, $g \in \GG$, acts on $S(1) = S^0$ by multiplication by $\det(g)$. 
Since $\E^{h\SG} \wedge S(1) \simeq \E^{h\SG} \wedge S^{\sigma-1}$ as $\mmu$-spectra, we have a
commutative diagram
\[
\xymatrix@C=40pt{
(\E^{h \SG} \wedge \ST)^{h\mmu} \ar[rr]^-{(\psi \wedge \det(\psi)-1)^{h\mmu}}\dto_\simeq &&
(\E^{h \SG} \wedge \ST)^{h\mmu}\dto^\simeq\\
\E_{-}^{h \SG} \ar[rr]^-{\det(\phi)\phi-1} && \E_{-}^{h \SG}.
}
\]
The result follows. 
\end{proof}

\begin{example}\label{ex:height-one!}
At height $1$, the determinant map $\GG \to \ZZ_p^\times$ is the identity.
We can also choose $\E =K$, the $p$-completion on complex theory. We have that
\[
K_\ast S^2 \cong K_\ast \langle \det \rangle
\]
so the $\K$-localization of $S^2$ is a valid model for the determinant sphere. If $p >2$, this must be the same as ours,
but at $p=2$ there is a possibility that $\Sdet \simeq S^2 \wedge P$, where $P = DQ$ is the dual of the `question mark complex'. By \cite{HMS_pic}, $P$ is the unique element in the $\Kn$-local
Picard group 
so that $K_\ast P \cong K_\ast S^0$ as $\ZZ_2^\times$-modules 
but $KO \wedge P \simeq \Sigma^{4}KO$. This possibility turns out to be the case. 

To see this, we observe that $KO= K^{h\mmu}$. We can use \Cref{thm:pretty-damn-cool} to deduce that $\ZZ_2^\times$-equivariantly, and therefore $\mmu$-equivariantly, we have an equivalence $K\wedge \Sdet \simeq K \wedge \ST$, where the action on the right hand side is diagonal. As mentioned above, $\mmu$-equivariantly $\ST $ is the representation sphere $S^{\sigma - 1}$, so we conclude that
\[ (K\wedge \Sdet)^{h\mmu} \simeq (K\wedge S^{\sigma-1})^{h\mmu}.  \]
By \Cref{prop:CommuteHFP}, we get that the left-hand side is $KO \wedge \Sdet$. For the right-hand side, we can use the $\mmu$-equivariant Bott periodicity equivalence $K\wedge S^{\sigma+1} \simeq K$ to conclude, altogether, that
\[ KO \wedge \Sdet  \simeq (K \wedge S^{-2})^{h\mmu} \simeq \Sigma^{-2} KO. \]
Thus $\Sdet \simeq S^2 \wedge P$.

Note that in this case we have shown that $\E_{-}^{h\SG} = \Sigma^{-2}KO$, and the fiber
sequence of \Cref{prop:fiber3} is a shifted version of that given by $P$ in \cite[Ex.~5.1]{GHM_pic}.
\end{example}

\bibliographystyle{amsalpha}
\bibliography{bib}

\end{document}